\documentclass[11pt]{amsart}

\usepackage[a4paper,hmargin=3.5cm,vmargin=3.5cm]{geometry}
\usepackage{amsfonts,amssymb,amscd,amstext}
\usepackage{graphicx}
\usepackage[dvips]{epsfig}
\usepackage{mathpazo}
\usepackage{enumerate}
\usepackage{titlesec}

\newcommand{\esca}[1]{\langle {#1} \rangle}
\newcommand{\doble}[1]{\ll \hspace{-0.1cm} {#1} \hspace{-0.1cm} \gg}
\newcommand{\curvo}[1]{\prec \hspace{-0.1cm} {#1} \hspace{-0.1cm} \succ}

\def\dist{{\rm dist}}
\def\Ccal{\mathcal{C}}
\def\Acal{\mathcal{A}}

\def\Vcal{\mathcal{V}}
\def\Dcal{\mathcal{D}}

\def\Ncal{\mathcal{N}}
\def\Scal{\mathcal{S}}
\def\Mcal{\mathcal{M}}
\def\Hcal{\mathcal{H}}
\def\Bcal{\mathcal{B}}

\def\Ocal{\mathcal{O}}
\def\Rcal{\mathcal{R}}
\def\Xcal{\mathcal{X}}
\def\Ycal{\mathcal{Y}}
\def\Lcal{\mathcal{L}}

\def\Wcal{\mathcal{W}}
\def\c{\mathbb{C}}

\def\R{\mathbb{R}}\def\r{\mathbb{R}}
\def\z{\mathbb{Z}}
\def\N{\mathbb{N}}\def\n{\mathbb{N}}
\def\s{\mathbb{S}}
\def\H{\mathbb{H}}\def\h{\mathbb{H}}
\def\k{\mathbb{K}}

\def\mgot{\mathfrak{m}}

\def\Mgot{\mathfrak{M}}

\def\Fgot{\mathfrak{F}}
\def\Pgot{\mathfrak{P}}
\def\Bsf{\mathsf{B}}
\def\Nsf{\mathsf{N}}

\def\Tsf{\mathsf{T}}

\def\csf{\mathsf{c}}
\def\esf{\mathsf{e}}

\titleformat{\section}
{\bfseries\large} {\thesection{.}}{0.2cm}{}
\titleformat{\subsection}
{\bfseries} {\thesubsection{.}}{0.15cm}{}
\titleformat{\subsubsection}
{\em}{\thesubsubsection{.}}{0.15cm}{}

\newtheorem{theorem}{Theorem}[section]

\newtheorem{claim}[theorem]{Claim}
\newtheorem{lemma}[theorem]{Lemma}
\newtheorem{corollary}[theorem]{Corollary}
\newtheorem{remark}[theorem]{Remark}
\newtheorem{definition}[theorem]{Definition}

\theoremstyle{definition}

\numberwithin{equation}{section}
\numberwithin{figure}{section}

\begin{document}

\thispagestyle{empty}

\vspace*{1cm}
\noindent{\bf\LARGE Compact complete null curves in Complex 3-space}

\vspace*{0.5cm}

\noindent{\large\bf Antonio Alarc\'{o}n and Francisco J. L\'{o}pez}

\footnote[0]{\vspace*{-0.4cm}

\noindent A. Alarc\'{o}n

\noindent Departamento de Geometr\'{\i}a y Topolog\'{\i}a, Universidad de Granada, E-18071 Granada, Spain

\noindent e-mail: {\tt alarcon@ugr.es}

\vspace*{0.1cm}

\noindent F.J. L\'{o}pez

\noindent Departamento de Geometr\'{\i}a y Topolog\'{\i}a, Universidad de Granada, E-18071 Granada, Spain

\noindent e-mail: {\tt fjlopez@ugr.es}

\vspace*{0.1cm}

\noindent Both authors are partially supported by MCYT-FEDER research project MTM2007-61775 and Junta de Andaluc\'{i}a Grant P09-FQM-5088}

\vspace*{1cm}

{\small
\noindent {\bf Abstract}\hspace*{0.1cm} We prove that for any open orientable surface $S$ of finite topology, there exist a  Riemann surface $\Mcal,$ a relatively compact domain $M\subset\Mcal$ and a continuous map $X:\overline{M}\to\c^3$ such that:
\begin{itemize}
\item $\Mcal $ and $M$ are homeomorphic to $S,$ $\Mcal-M$ and $\Mcal-\overline{M}$ contain no relatively compact components in $\Mcal,$
\item $X|_M$ is a complete null holomorphic curve, $X|_{\overline{M}-M}:\overline{M}-M\to\c^3$ is an embedding and the Hausdorff dimension of $X(\overline{M}-M)$ is  $1.$
\end{itemize}

Moreover, for any $\epsilon>0$ and  compact null holomorphic curve $Y:N\to\c^3$ with non-empty boundary $Y(\partial N),$ there exist Riemann surfaces $M$ and $\Mcal$  homeomorphic to $N^\circ$  and a map $X:\overline{M}\to\c^3$ in the above conditions such that $\delta^H(Y(\partial N),X(\overline{M}-M))<\epsilon,$ where $\delta^H(\cdot,\cdot)$ means Hausdorff distance in $\c^3.$

\vspace*{0.1cm}

\noindent{\bf Keywords}\hspace*{0.1cm} Null holomorphic curves, Calabi-Yau problem, Plateau problem.

\vspace*{0.1cm}

\noindent{\bf Mathematics Subject Classification (2010)}\hspace*{0.1cm} 53C42, 32H02, 53A10.
}


\section{Introduction}

Given an open Riemann surface $\Ncal,$ a null (holomorphic) curve in $\c^3$ is a holomorphic immersion $X=(X_j)_{j=1,2,3}:\Ncal\to\c^3$ satisfying that $\sum_{j=1}^3 (dX_j)^2=0,$ where $d$ is the complex differential.
In this paper we deal with the existence of compact complete null curves in $\c^3$ accordingly to the following 
\begin{definition}\label{def:compact}
Let $M$ be a relatively compact domain in an open Riemann surface. A continuous map $X:\overline{M}\to\c^3$ is said to be a compact complete null curve in $\c^3$ if $X|_M$ is a complete null curve.
\end{definition}
The first approach to this matter was made by Mart\'{i}n and Nadirashvili in the context of simply connected minimal surfaces in $\r^3$ \cite{MN}. The corresponding finite topology case was considered in \cite{Al}, see also \cite{Na2,AN} for other related questions. Compact complete minimal surfaces manifest the interplay between two well studied topics on minimal surface theory: Plateau's and Calabi-Yau's problems. The first one consists of finding a compact minimal surface spanning a given closed curve in $\r^3,$ and it was independently solved by Douglas \cite{Do} and Rad\'{o} \cite{Ra}.  The second one deals with the construction of complete minimal surfaces in $\r^3$ with bounded coordinate functions, and  the most relevant examples were given by Jorge-Xavier \cite{JX} and Nadirashvili \cite{Na}.

In the very last few years, the study of the Calabi-Yau problem has evolved in the direction of null curves in $\c^3.$  Observe that a holomorphic map $X:\Ncal\to\c^3$ is a null curve if and only if both $\Re(X):\Ncal\to\r^3$ and $\Im(X):\Ncal\to\r^3$ are conformal minimal immersions. Furthermore, the Riemannian metric on $\Ncal$ induced by $X$ is twice the one induced by $\Re(X)$ and $\Im(X).$ Therefore,  complete bounded null curves in $\c^3$ provide  complete bounded minimal surfaces in $\r^3$ with well defined and bounded conjugate surface, and viceversa. Existence of a vast family of complete bounded null curves with arbitrary topology and other additional properties is known \cite{MUY1,AL2}.

If $M$ is a relatively compact domain in an open Riemann surface  and $F:\overline{M}\to\r^3$ is a continuous map such that $F|_{M}:M\to\r^3$ is a conformal complete minimal immersion whose conjugate immersion $(F|_M)^*:M\to\r^3$ is well defined, then $(F|_M)^*$ does not necessarily extend as a continuous map to $\overline{M}.$ Therefore, it is natural to wonder whether there exist compact complete null curves in $\c^3.$

In this paper we answer this question, proving considerably more. 

\begin{quote}
{\bf Main Theorem.} 
{\em Let $S$ be an open orientable surface of finite topology. 

Then there exist a  Riemann surface $\Mcal,$ a relatively compact domain $M\subset\Mcal$ and a continuous map $X:\overline{M}\to\c^3$ such that:
\begin{itemize}
\item $\Mcal $ and $M$ are homeomorphic to $S,$ $\Mcal-M$ and $\Mcal-\overline{M}$ contain no relatively compact components in $\Mcal,$
\item  $X|_{M}$ is a complete null curve, $X|_{\overline{M}-M}:\overline{M}-M\to\c^3$ is an embedding and the Hausdorff dimension of $X(\overline{M}-M)$ is $1.$
\end{itemize}
 
Moreover, for any $\epsilon>0$ and for any compact null curve $Y:N\to\c^3$ with non-empty boundary $Y(\partial N),$ there exist Riemann surfaces $M$ and $\Mcal$  homeomorphic to $N^\circ$  and a map $X:\overline{M}\to\c^3$ in the above conditions such that $\delta^H(Y(\partial N),X(\overline{M}-M))<\epsilon,$ where $\delta^H(\cdot,\cdot)$ means Hausdorff distance in $\c^3.$}
\end{quote}

Unfortunately, the techniques we use do not give enough control over the topology of $\overline{M}- M$ to assert that it consists, for instance,  of a finite collection of Jordan curves. Concerning the second part of the theorem, it is worth mentioning that there exist Jordan curves in $\c^3$ which are not spanned by any null curve.  Main Theorem actually follows from a more general density  result (see Theorem \ref{th:fun}). Some other similar existence results for complex curves in $\c^2,$ null holomorphic curves in the complex Lie group ${\rm SL}(2,\c),$ Bryant surfaces in hyperbolic 3-space $\h^3$ and minimal surfaces in $\r^3$  can be also derived from it, see Corollary \ref{co:consecuencias}. See \cite{MUY1,MUY2,AL2} for related results.

\subsection*{Acknowledgements}
Part of this work was done when the first author was visiting the Institut de Math\'{e}matiques de Jussieu in Paris. He wish to thank this institution, and in particular Rabah Souam, for the kind invitation and hospitality, and for providing excellent working conditions.


\section{Preliminaries}

We denote by $\|\cdot\|$ and $\langle\cdot,\cdot\rangle$ the Euclidean norm and metric in $\k^n,$ where $\k=\r$ or $\c,$ and for any compact topological space $K$ and continuous map $f:K \to \k^n$ we denote by $\|f\|=\max\{\|f(p)\|\;|\; p \in K\}$ the maximum norm of $f$ on $K.$ We also denote by $\imath=\sqrt{-1}.$

\subsection{Riemann surfaces}
Given a Riemann surface $\Mcal$ with non empty boundary, we denote by $\partial \Mcal$ the (possibly non-connected) 1-dimensional topological manifold determined by its boundary points. For any  $A \subset \Mcal,$ we denote by $A^\circ,$ $\overline{A}$ and $Fr (A)=\overline{A}-A^\circ$ the interior of $A$ in $\Mcal,$ the closure of $A$ in $\Mcal$ and the topological frontier of $A$ in $\Mcal.$ Open connected subsets of $\Mcal$ will be called {\em domains}, and those proper topological subspaces of $\Mcal$ being Riemann surfaces with boundary are said to be {\em regions}. 

A Riemann surface $\Mcal$ is said to be a  {\em bordered Riemann surface} if it is  compact, $\partial M \neq \emptyset$  and $\partial \Mcal$ is smooth. A compact region of $\Mcal$ is said to be a {\em bordered region} if it is a bordered Riemann surface.

\begin{remark} \label{re:prime}
Throughout this paper, $\Rcal_0$ and $\omega$ will denote a fixed but arbitrary bordered Riemann surface and a smooth conformal metric on $\Rcal_0,$ repectively. We call $\Rcal$ the open Riemann surface $\Rcal_0-\partial\Rcal_0,$ and write $\esf\in\n$ for the number of ends of $\Rcal$ (or equivalently, for the number of connected components of $\partial \Rcal_0$).
\end{remark}

A subset $A \subset \Rcal$ is said to be {\em Runge} if $j_*:\Hcal_1(A,\z) \to \Hcal_1(\Rcal_0,\z)$ is injective, where $j:A \to \Rcal_0$ is the inclusion map. If $A$ is a compact region in $\Rcal,$ this simply means that  $\Rcal-A$ has no relatively compact components in $\Rcal.$

In the remainder of this subsection we introduce the necessary notations for a precise statement of the approximation result by null curves in Lemma \ref{lem:runge}, which is the starting point of the paper.

A 1-form $\theta$ on a subset $S\subset \Rcal$ is said to be of type $(1,0)$ if for any conformal chart $(U,z)$ in $\Rcal,$ $\theta|_{U \cap S}=h(z) dz$ for some function $h:U \cap S \to \overline{\c}.$  
%

\begin{definition}

Assume that $S\subset \Rcal$ is compact. We say that
\begin{itemize}
\item a function $f:S \to \c$ can be uniformly approximated on $S$ by holomorphic functions in $\Rcal$ if and only if there exists a sequence of holomorphic functions $\{f_n:\Rcal\to\c\}_{n \in \n}$ such that $\{f_n-f\}_{n \in \n} \to 0$ uniformly on $S,$ and
\item a 1-form $\theta$ on $S$ can be uniformly approximated on $S$ by holomorphic 1-forms in $\Rcal$ if and only if there exists a sequence of holomoprhic 1-forms $\{\theta_n\}_{n \in \n}$ on $\Rcal$ such that $\{\frac{\theta_n-\theta}{dz}\}_{n \in \n} \to 0$ uniformly on $S \cap U,$ for any closed conformal disc $(U,dz)$ on $\Rcal.$
\end{itemize}
\end{definition}

The following definition is crucial in our arguments.

\begin{definition}[Admissible set]
A compact subset $S\subset\Rcal$ is said to be admissible if and only if
\begin{itemize}
\item $S$ is Runge,
\item $M_S:=\overline{S^\circ}$ is a finite collection of pairwise disjoint compact regions in $\Rcal$ with $\mathcal{C}^0$ boundary,
\item $C_S:=\overline{S-M_S}$ consists of a finite collection of pairwise disjoint analytical Jordan arcs, and
\item any component $\alpha$ of $C_S$  with an endpoint  $P\in M_S$ admits an analytical extension $\beta$ in $\Rcal$ such that the unique component of $\beta-\alpha$ with endpoint $P$ lies in $M_S.$
\end{itemize}
\end{definition}
\begin{figure}[ht]
    \begin{center}
    \scalebox{0.35}{\includegraphics{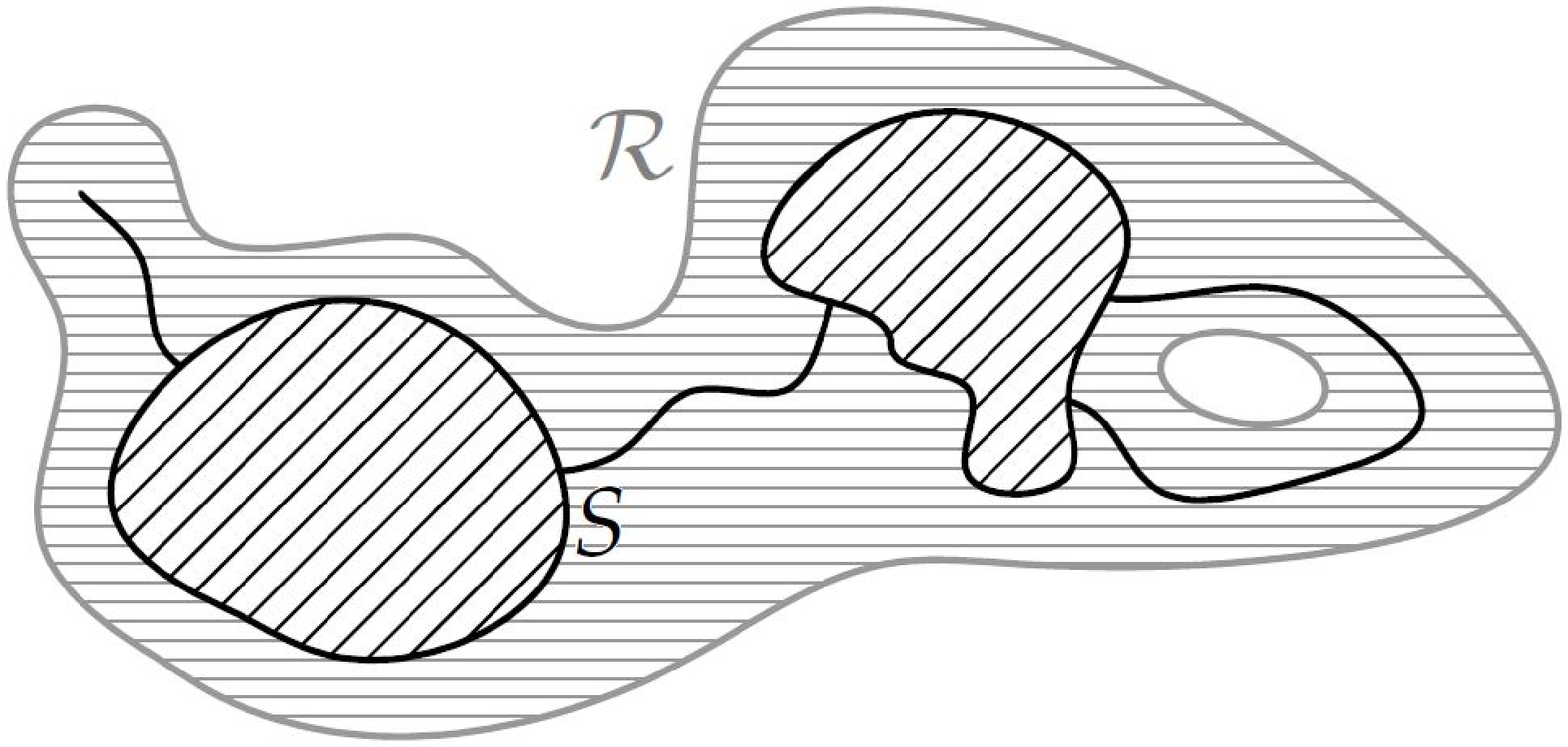}}
        \end{center}
        \vspace{-0.5cm}
\caption{An admissible set $S$ in $\Rcal.$}\label{fig:admi}
\end{figure}

The next one deals with the notion of smoothness of functions and 1-forms on admissible subsets. 
\begin{definition}
Assume that  $S$ is an admissible subset of $\Rcal.$ 
\begin{itemize}
\item A function $f:S \to \c$ is said to be smooth if and only if $f|_{M_S}$  admits a smooth extension $f_0$ to an open domain $V$ in $\Rcal$ containing $M_S,$ and for any component $\alpha$ of $C_S$ and any open analytical Jordan arc $\beta$ in $\Rcal$ containing $\alpha,$  $f$ admits a smooth extension $f_\beta$ to $\beta$ satisfying that $f_\beta|_{V \cap \beta}=f_0|_{V \cap \beta}.$
\item A 1-form $\theta$ on $S$ is said to be smooth if and only if $(\theta|_{S\cap U})/dz$ is  a smooth function for any closed conformal disk $(U,z)$ on $\Rcal$ such that $S\cap U$ is an admissible set.
\end{itemize}
\end{definition}

Given a smooth function $f:S \to \c$ on an admissible $S \subset \Rcal,$ we denote by $df$ the 1-form of type (1,0) given by 
$df|_{M_S}=d (f|_{M_S})$ and $df|_{\alpha \cap U}=(f \circ \alpha)'(x)dz|_{\alpha \cap U},$
where $(U,z=x+\imath y)$ is any conformal chart on $\Rcal$ satisfying that $z(\alpha \cap U)\subset \r.$ 
Then $df$ is well defined  and smooth. 
The $\Ccal^1$-norm of $f$ on $S$ is given by
$$\|f\|_1=\max_S\{\|f\|+\|df/\omega\|\}.$$
A sequence of smooth functions $\{f_n\}_{n \in \n}$ on $S$  is said to converge in the $\Ccal^1$-topology to a smooth function $f$ on $S$ if $\{\|f-f_n\|_1\}_{n \in \n}\to 0.$ If in addition $f_n$ is (the restriction to $S$ of) a holomorphic function on $\Rcal$ for all $n,$  we also say that $f$ can be uniformly $\Ccal^1$-approximated on $S$ by holomorphic functions on $\Rcal.$

Likewise one can define the notions of smoothness, (vectorial) differential,  $\Ccal^1$-norm and uniform $\Ccal^1$-approximation for maps $f:S \to \c^k,$ $k \in \n,$ $S\subset \Rcal$ admissible.

Let $S$ be an admissible subset of $\Rcal,$ and let $W$ be a Runge open subset of $\Rcal$ whose closure $\overline{W}$ in $\Rcal_0$ is a compact region and  $j_*:\Hcal_1(W,\z) \to \Hcal_1(\overline{W},\z)$ is an isomorphism, where $j:W \to \overline{W}$ is the inclusion map. $W$ is said to be a {\em tubular neighborhood} of $S$ if  $S \subset W,$  
$(j_0)_*:\Hcal_1(S,\z) \to \Hcal_1(\overline{W},\z)$ is an isomorphism  and $\chi(W-S)=\chi(\overline{W}-S)=0,$ where  $j_0:S \to \overline{W}$ is the inclusion map and  $\chi(\cdot)$ means Euler characteristic. In particular, $W-S$ (respectively, $\overline{W}-S^\circ\subset\Rcal_0$) consists of a family of  pairwise disjoint open (respectively, compact) annuli.
\begin{figure}[ht]
    \begin{center}
    \scalebox{0.35}{\includegraphics{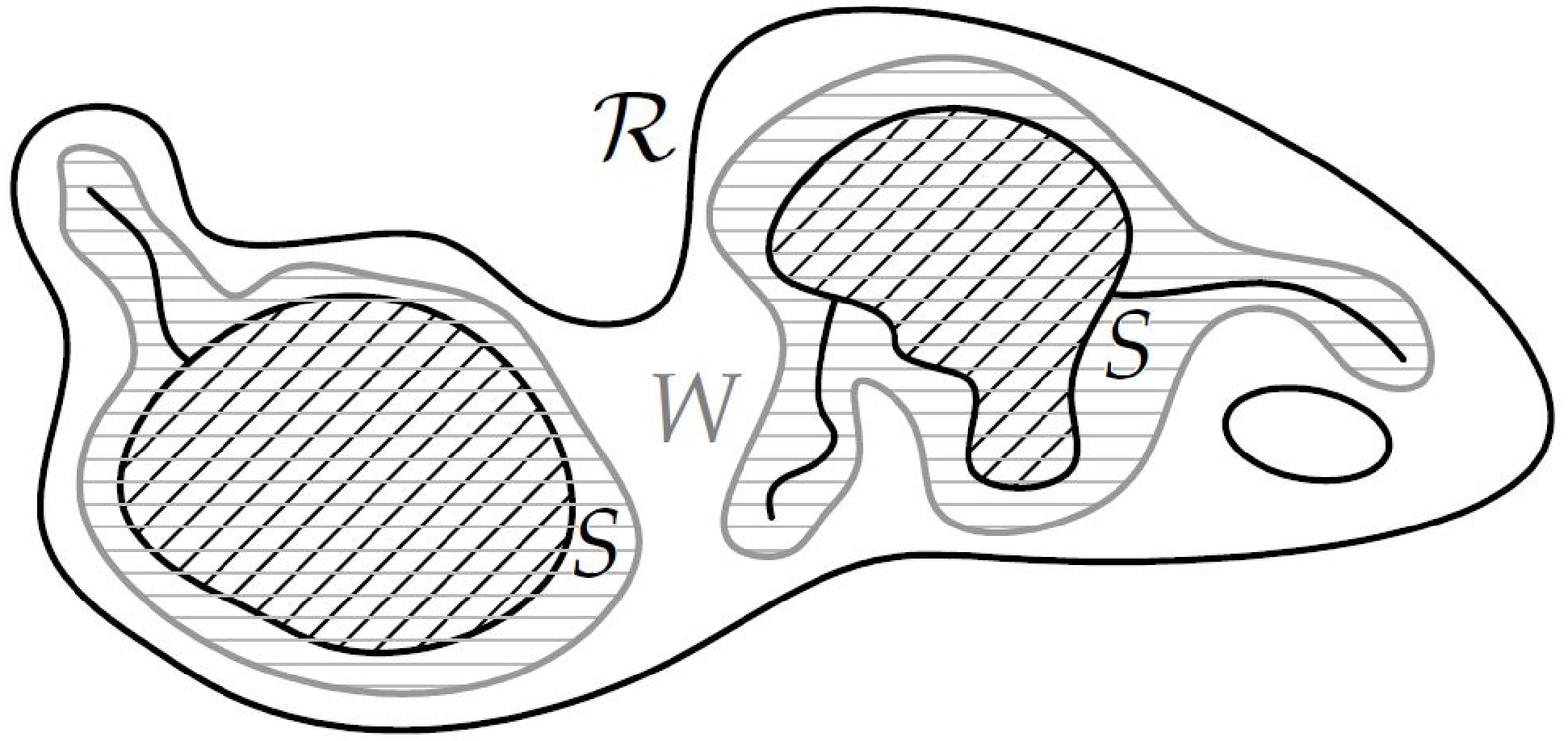}}
        \end{center}
        \vspace{-0.5cm}
\caption{A tubular neighborhood $W$ of an admissible set $S$ in $\Rcal.$}\label{fig:tubular}
\end{figure}

For instance, assume that $S$ is a finite family of pairwise disjoint smooth Jordan curves $\gamma_j,$ $j=1,\ldots,k,$ in $\Rcal,$ take $\epsilon>0$ and set 
$W=\{p \in \Rcal\,|\; \dist_\omega(p,S)<\epsilon\},$ where $\dist_\omega$ means Riemannian distance in $(\Rcal_0,\omega).$ If $\epsilon$ is small enough, the exponential map $F:S \times [-\epsilon,\epsilon] \to \overline{W},$ $F(p,t)=exp_p(t n(p)),$ is a diffeomorphism  and $W=F(S \times (-\epsilon,\epsilon)),$ where $n$ is a normal field along $S$ in $(\Rcal,\omega).$ In this case,  $W$ is said to be  a {\em metric tubular neighborhood} of $S$ (of radious $\epsilon$). Furthermore, if $\pi_1:S \times (-\epsilon,\epsilon) \to S$ denotes the projection $\pi_1(p,t)=p,$ we denote by $\Pgot:W \to S,$  $\Pgot(q):=\pi_1(F^{-1}(q)),$ the natural orthogonal projection. 

\begin{definition}
We denote by $\Bsf(\Rcal)$ the family of Runge bordered regions $M\subset\Rcal$ such that $\Rcal$ is a  tubular neighborhood of $M.$ Given $M,N\in\Bsf(\Rcal),$ we say that $M<N$ if and only if $M\subset N^\circ.$
\end{definition}

\subsection{Null curves in $\c^3$}

Throughout this paper we adopt column notation for both vectors and matrices of linear transformations in $\c^3,$ and make the convention  $$\c^3 \ni (z_1,z_2,z_3)^T \equiv (\Re(z_1),\Im(z_1), \Re(z_2),\Im(z_2),\Re(z_3),\Im(z_3))^T \in \r^6\quad \big( \text{$(\cdot)^T$ means ``transpose"} \big).$$

Let us start this subsection by introducing some operators which are strongly related to the geometry of $\c^3$ and null curves.

\begin{definition}
Let $A\subset\c^3.$ We denote by
\begin{itemize}
\item $\doble{\cdot,\cdot}:\c^3\times\c^3\to \c,$ $\doble{ u,v }=\bar{u}^T \cdot v,$ the usual Hermitian inner product in $\c^3,$
\item $\doble{ A}^\bot=\{v\in\c^3\,|\, \doble{ u,v}=0 \, \forall u \in A\},$
\item $\langle\cdot,\cdot\rangle=\Re(\doble{\cdot,\cdot})$ the Euclidean scalar product of $\c^3\equiv\r^6,$ 
\item $\langle A\rangle^\bot=\{v\in\c^3\,|\, \langle u,v \rangle=0 \, \forall u \in A\},$ 
\item $\curvo{\cdot,\cdot}:\c^3\times\c^3\to\c,$ $\curvo{u,v} = u^T\cdot v,$ and 
\item $\curvo{A}^\bot=\{v\in\c^3\,|\,\curvo{  u, v}=0\,\forall u\in A\}.$
\end{itemize}
\end{definition}
Notice that $\curvo{ \overline{u}}^\bot= \doble{ u}^\bot\subset\esca{u}^\bot$ for all $u\in\c^3,$ and the equality holds if and only if $u=0\in\c^3$. 

\begin{definition} 
A vector $u \in \c^3-\{0\}$ is said to be null if and only if $\curvo{u,u}=0.$ We denote by $\Theta=\{u \in \c^3-\{0\} \,|\, u\text{ \em is null}\}.$

A basis $\{u_1,u_2,u_3\}$ of $\c^3$ is said to be $\curvo{\cdot,\cdot}$-conjugate if and only if $\curvo{ u_j,u_k}=\delta_{jk},$ $j,k\in\{1,2,3\}.$
\end{definition}

We denote by $\Ocal(3,\c)$ the complex orthogonal group $\{A\in\Mcal_3(\c)\,|\, A^T\cdot A=I_3\},$ or in other words, the group of matrices  whose column vectors determine a $\curvo{\cdot,\cdot}$-conjugate basis of $\c^3.$ We also denote by $A:\c^3 \to \c^3$ the complex linear transformation induced by $A \in \Ocal(3,\c).$ 

It is clear that $A(\Theta)=\Theta$ for all $A \in \Ocal(3,\c).$ 

Let $\Ncal$ be an open Riemann surface.

\begin{definition}\label{def:null curve}
A holomorphic map $X:\Ncal\to\c^3$ is said to be a null curve if and only if $\curvo{dX,dX}=0$ and $\doble{dX,dX}$ never vanishes on $\Ncal.$ 

Conversely, given an exact holomorphic vectorial 1-form $\Phi$ on $\Ncal$ satisfying that $\curvo{\Phi,\Phi}=0$ and $\doble{\Phi,\Phi}$ never vanishes on $\Ncal,$ then the map $X:\Ncal\to\c^3,$ $X(P)=\int^P \Phi,$ defines a null curve in $\c^3.$ 
\end{definition}

If $X:\Ncal\to\c^3$ is a null curve then the pull back metric $X^*(\esca{\cdot,\cdot})$ on $\Ncal$ is determined by the expresion $\esca{dX,dX}=\doble{dX,dX}.$ 
Given two subsets $V_1, V_2\subset \Ncal,$ we denote by $\dist_{(\Ncal,X)}(V_1,V_2)$ the distance between $V_1$ and $V_2$ in the Riemannian surface  $(\Ncal,X^*(\esca{\cdot,\cdot})).$

\begin{remark}\label{rem:AF}
Let $X:\Ncal\to\c^3$ be a null curve and $A=(a_{jk})_{j,k=1,2,3}\in\Ocal(3,\c).$ Then $A\circ X:\Ncal\to\c^3$ is a null curve as well and $X^*\esca{\cdot,\cdot}\geq \frac1{ \|A\|^2} (A\circ X)^*\esca{\cdot,\cdot},$ where $\|A\|\,= \big(\sum_{j,k}  |a_{jk}|^2  \big)^{1/2}.$
\end{remark}

\begin{definition}
Given a subset $S\subset \Rcal$ we denote by $\Nsf(S)$ the space of maps $X:S \to\c^3$ extending as a null curve to an open neighborhood of $S$ in $\Rcal.$ 
\end{definition}

\begin{definition}\label{def:gen-null}
Let $S\subset \Rcal$ be an admissible subset.
A smooth map $X:S\to\c^3$ is said to be a generalized null curve in $\c^3$ if and only if
 $X|_{M_S}\in \Nsf(M_S),$ $\curvo{dX,dX}=0$ on $S$ and $\doble{dX,dX}$ never vanishes on $S.$
\end{definition}

The following technical lemma is a key tool in this paper.

\begin{lemma}[Approximation Lemma \cite{AL,AL2}]\label{lem:runge}
Let $S\subset\Rcal$ be a connected admissible set and let $X=(X_j)_{j=1,2,3}:S \to \c^3$ be a generalized null curve.

Then $X$ can be uniformly $\Ccal^1$-approximated on $S$ by a sequence of null curves $\{Y_n=(Y_{j,n})_{j=1,2,3}:\Rcal\to\c^3\}_{n\in\N}.$ In addition, we can choose $Y_{3,n}=X_3$ $\forall n\in\N$ provided that $X_3$ extends holomorphically to $\Rcal$ and $dX_3$ never vanishes on $C_S.$
\end{lemma}


\section{Completeness Lemma}\label{sec:CL}

In this section we state and prove the technical result which is the kernel of this paper (Lemma \ref{lem:CL} below). Its proof  requires of the approximation result by null curves in Lemma \ref{lem:runge}. Lemma \ref{lem:CL} has more strength than similar results in \cite{MN,Al}, and its proof presents some differences.

\begin{lemma}\label{lem:CL}
Let $M,$ $N\in\Bsf(\Rcal)$ with $N<M,$ let $\Tsf$ be a metric tubular neighborhood of $\partial M$ in $\Rcal$ and denote by $\Pgot:\Tsf\to\partial M$ the orthogonal projection. 
Consider $X\in\Nsf(M),$ an analytical map $\Fgot:\partial M\to\c^3,$ and $\mu>0$ so that
\begin{equation}\label{eq:CL}
\|X-\Fgot\|<\mu\quad\text{on }\partial M.
\end{equation}

Then, for any $\rho>0$ and $\epsilon>0,$ there exist $\Mcal\in\Bsf(\Rcal)$ and $\Xcal\in\Nsf(\Mcal)$ such that
\begin{enumerate}[{\rm (\ref{lem:CL}.a)}]
\item $N<\Mcal<M$ and $\partial \Mcal\subset\Tsf$, 
\item $\Xcal|_{\partial\Mcal}:\partial\Mcal\to\c^3$ is an embedding,
\item $\rho<\dist_{(\Mcal,\Xcal)}(N,\partial\Mcal),$
\item $\|\Xcal-X\|_1<\epsilon$ on $N,$ and
\item $\|\Xcal-\Fgot\circ\Pgot\|<\sqrt{4\rho^2+\mu^2}+\epsilon$ on $\partial\Mcal.$
\end{enumerate}
\end{lemma}

Roughly speaking, the above Lemma asserts that a compact null curve $X(M)$ in $\c^3$ can be perturbed into another compact null curve $\Xcal(\Mcal)$ with larger intrinsic radious. This deformation hardly modifies the null curve in a prescribed compact set and, in addition, the effect of the deformation in the boundaries of the null curves is quite controlled. The bounds $\mu,$ $\rho$ and $\sqrt{4\rho^2+\mu^2}+\epsilon$ in \eqref{eq:CL} and Items (\ref{lem:CL}.c) and (\ref{lem:CL}.e) follow the spirit of Nadirashvili's original construction \cite{Na} (see \eqref{eq:pitag} below).

The null curve $\Xcal$ which proves the lemma will be obtained from $X$ after two different perturbation processes. In the first one, which is enclosed in Claim \ref{cl:antena} below, the effect of the deformation is strong over a family of Jordan arcs in $M$ (see the definition of the arcs $r_{i,j}$ in Items (C.1), (C.2) and (C.3)) and slight on a compact region containing $N.$ In the second stage, see Claim \ref{cl:2a}, the deformation mainly acts on a family of compact discs (see the definition of $K_{i,j}$ in (E.2)) and hardly works on a compact region containing $N.$ 

\subsection{Proof of Lemma \ref{lem:CL}}
Take $\epsilon_0>0$ to be specified later.

By \eqref{eq:CL} and a continuity argument, for any $P\in \partial M$ there exists a simply connected open neighborhood $V_P$ of $P$ in $M\cap \Tsf$ such that
\begin{enumerate}[{\rm ({A}.1)}]
\item $\Pgot(Q)\in V_P\cap \partial M,$ $\forall Q\in V_P,$
\item $\max\left\{\|X(Q_1)-X(Q_2)\|\,,\,\|\Fgot(\Pgot(Q_1))-\Fgot(\Pgot(Q_2))\|\right\}<\epsilon_0,$ $\forall \{Q_1,Q_2\}\subset V_P,$ and
\item $\|X-\Fgot\circ\Pgot\|<\mu$ on $V_P.$
\end{enumerate}

Set $\Vcal=\{V_P\}_{P\in\partial M},$ and observe that $\Vcal\cap \partial M:=\{V_P\cap \partial M\}_{P\in\partial M}$ is an open covering of $\partial M.$ Take $M_0\in\Bsf(\Rcal)$ satisfying that
\begin{equation}\label{eq:M_0} 
N<M_0<M\quad\text{and}\quad M-M_0^\circ \subset \bigcup_{P\in\partial M}V_P,
\end{equation}
and note that $\Vcal\cap (M-M_0^\circ):=\{V_P\cap (M-M_0^\circ)\}_{P\in\partial M}$ is an open covering of $M-M_0^\circ$ as well.  
One has that $M-M_0^\circ=\cup_{i=1}^\esf A_i,$ where $\{A_i\}_{i=1}^\esf$ are pairwise disjoint compact annuli. Write $\alpha_i\subset\partial M_0$ and $\beta_i\subset\partial M$ for the two components of  $\partial A_i,$  $i=1,\ldots,\esf.$ For each $m\in\n,$ let $\z_m$ denote the additive cyclic group of integers modulus $m.$ Since $\Vcal\cap (M-M_0^\circ)$ is an open covering of $A_i,$ then there exist $\mgot\in\n,$ $\mgot\geq 3,$ and a collection of compact Jordan arcs $\{\alpha_{i,j}\;|\; (i,j)\in \{1,\dots,\esf\}\times\z_\mgot\}$ satisfying that
\begin{enumerate}[{\rm ({B}.1)}]
\item $\cup_{j\in\z_\mgot} \alpha_{i,j}=\alpha_i,$
\item $\alpha_{i,j}$ and $\alpha_{i,j+1}$ have a common endpoint $Q_{i,j}$ and are otherwise disjoint for all $j\in\z_\mgot,$ and
\item $\alpha_{i,j}\cup\alpha_{i,j+1}\subset V_{i,j}\in\Vcal$ for all $j\in\z_\mgot,$
\end{enumerate}
$\forall i\in\{1,\ldots,\esf\}.$

For any $(i,j)\in\{1,\dots,\esf\}\times\z_\mgot$ fix $P_{i,j}\in V_{i,j-1}\cap V_{i,j}$ and $e_{i,j}\in\s^5-\Theta$ such that
\begin{equation}\label{eq:eij}
\left\|\Pi_{i,j}\big(X(P_{i,j})-\Fgot(\Pgot(P_{i,j}))\big)\right\|<\epsilon_0,
\end{equation}
where $\Pi_{i,j}:\c^3\to\doble{e_{i,j}}^\bot$ is the orthogonal projection. This choice is possible since $\s^5-\Theta$ is dense in $\s^5.$
Label 
\[
w_{i,j}=\frac{\overline{e_{i,j}}}{\curvo{\overline{e_{i,j}},\overline{e_{i,j}}}},
\]
for all $i,j,$ and notice that 
\begin{equation}\label{eq:orto}
\doble{e_{i,j}}^\bot=\curvo{w_{i,j}}^\bot.
\end{equation}

Since $w_{i,j}\notin\Theta,$ then we can take $u_{i,j},$ $v_{i,j} \in \curvo{w_{i,j}}^\bot$ such that  $\{u_{i,j},v_{i,j},w_{i,j}\}$ is a $\curvo{\cdot,\cdot}$-conjugate basis of $\c^3.$  Denote by $\Acal_{i,j}$ the complex orthogonal matrix $(u_{i,j},v_{i,j},w_{i,j})^{-1},$ 
for all $i\in \{1,\ldots,\esf\}$ and $j\in \z_\mgot.$ 

Let $\{r_{i,j}\;|\;j\in\z_\mgot\}$ be a family of pairwise disjoint analytical compact Jordan arcs in $A_i$ such that
\begin{enumerate}[{\rm ({C}.1)}]
\item $r_{i,j}\subset V_{i,j-1}\cap V_{i,j}\cap V_{i,j+1}$ (see (A.1), (B.2) and (B.3)),
\item $r_{i,j}$ has initial point $Q_{i,j},$ final point $\Pgot(Q_{i,j})$ and it is otherwise disjoint from $\alpha_i\cup\beta_i,$ and
\item $S=M_0\cup(\cup_{i,j}r_{i,j})$ is admissible. 
\end{enumerate} 

As we announced above, the null curve $\Xcal$ will be obtained from $X$ after two deformation procedures. The first one strongly works in $\cup_{i,j}r_{i,j}$ (see property (D.2) below) and is enclosed in the following

\begin{claim}\label{cl:antena}
There exists $H\in{\sf N}(M)$ such that, for any $(i,j) \in \{1,\ldots,\esf\} \times \z_\mgot,$
\begin{enumerate}[{\rm ({D}.1)}]
\item $\|H-X\|<\epsilon_0$ on $S,$
\item if $J \subset r_{i,j}$ is a Borel measurable subset, then 
\begin{multline*} \label{eq:comple2}
\min\{\Lcal_\c((\Acal_{i,j} \cdot H|_{J})_3),\Lcal_\c((\Acal_{i,j+1} \cdot H|_{J})_3)\}+\\ \min\{\Lcal_\c((\Acal_{i,j} \cdot H|_{r_{i,j}-J})_3),\Lcal_\c((\Acal_{i,j+1} \cdot H|_{r_{i,j}-J})_3)\} > 2\rho\cdot\max\{\|\Acal_{i,j}\|,\|\Acal_{i,{j+1}}\|\}, 
\end{multline*}
where $(\,\cdot\,)_3$ means third (complex) coordinate and $\Lcal_\c(\cdot)$ Euclidean length in $\c,$ and
\item $\|H-X\|_1<\epsilon_0$ on $M_0.$
\end{enumerate}
\end{claim}
\begin{proof} 
Let $r_{i,j}(u),$ $u\in [0,1],$ be a smooth parameterization of $r_{i,j}$ with $r_{i,j}(0)=Q_{i,j}$ and $r_{i,j}(1)=\Pgot_i(Q_{i,j}).$ From (A.2) and (C.1) there exists a non-empty open subset $\Xi_{i,j}\subset\c^3$ satisfying that
\begin{equation}\label{eq:peque}
X(V_{i,j-1}\cap V_{i,j})\subset \Xi_{i,j}\quad\text{and}\quad \|w_1-w_2\|<\epsilon_0\; \forall \{w_1,w_2\}\subset\Xi_{i,j}.
\end{equation}
Then (B.3), \eqref{eq:peque} and (C.1) give that
$X(r_{i,j-1} \cup \alpha_{i,j}\cup r_{i,j}\cup  \alpha_{i,j+1} \cup r_{i,j+1})\subset \Xi_{i,j}.$
Consider $\lambda_{i,j} \in \Theta$ such that $(\Acal_{i,j}(\lambda_{i,j}))_3,$ $(\Acal_{i,j+1}(\lambda_{i,j}))_3\neq 0$ and 
\begin{equation}\label{eq:a3}
\{\Lambda_{i,j}(s):=X(Q_{i,j})+s \lambda_{i,j}\,|\, s \in [0,1]\}\subset \Xi_{i,j}.
\end{equation}
Set $\Lambda_{i,j}^*(s)=\Lambda_{i,j}(1-s),$ $s \in [0,1].$ Take $n \in \n$ large enough so that 
\begin{equation}\label{eq:a2}
n\cdot \min\{|(\Acal_{i,j}(\lambda_{i,j}))_3|,|(\Acal_{i,j+1}(\lambda_{i,j}))_3|\}>2\rho\cdot\max\{\|\Acal_{i,j}\|,\|\Acal_{i,j+1}\|\}.
\end{equation}

Set $d_{i,j}:[0,1] \to \c^3,$ 
\begin{itemize}
\item $d_{i,j}(u)=\Lambda_{i,j}(n u-b+1)$ if $u \in [\frac{b-1}{n},\frac{b}{n}]$ and  $b\in \{1,\ldots,n\}$ is odd, and
\item $d_{i,j}(u)=\Lambda_{i,j}^*(n u-b+1)$ if $u \in [\frac{b-1}{n},\frac{b}{n}]$ and  $b\in \{1,\ldots,n\}$ is even.
\end{itemize}

Notice that the curves $d_{i,j}$ are continuous,  weakly differentiable and satisfy that $\curvo{d_{i,j}'(u),d_{i,j}'(u)}=0.$ Up to replacing $H|_{r_{i,j}}$ for $d_{i,j}$ for all $i,$ $j,$ items (D.1), (D.2) and (D.3) formally hold. To finish, approximate $d_{i,j}$ by a smooth curve $c_{i,j}$ matching smoothly with $X$ at $Q_{i,j},$ and so that the map $\tilde{H}:S \to \c^3$  given by $\tilde{H}|_{M}=X,$ $\tilde{H}|_{r_{i,j}}(u)=c_{i,j}(u)$ for all $u \in [0,1],$ $i$ and $j,$ is a generalized null curve satisfying all the above items. Indeed, if $c_{i,j}$ is chosen close enough to $d_{i,j},$ (D.1) follows from \eqref{eq:a3} and (D.2) from \eqref{eq:a2}. 
Apply Lemma \ref{lem:runge} to  $\tilde{H}$ and we are done. 
\end{proof}

Denote by $\Omega_{i,j}$ the closed disc in $M-M_0^\circ$ bounded by $\alpha_{i,j},$ $r_{i,j-1},$ $r_{i,j}$ and a piece, named $\beta_{i,j},$ of $\beta_i$ connecting $\Pgot(Q_{i,j-1})$ and $\Pgot(Q_{i,j}).$ Since $V_{i,j}$ is simply connected, then (A.1), (B.3) and (C.1) give that
\begin{equation}\label{eq:Oin}
\Omega_{i,j}\cup\Omega_{i,j+1}\subset V_{i,j}.
\end{equation}
Consider simply-connected compact neighborhoods $\tilde{\alpha}_{i,j}$  and $\tilde{r}_{i,j}$ in $M-M_0^\circ$ of $\alpha_{i,j}$ and $r_{i,j},$ respectively,  $i=1,\ldots,\esf,$ $j \in \z_\mgot,$  satisfying the following properties:
\begin{enumerate}[({E}.1)]
\item $\tilde{\alpha}_{i,j}\cap \tilde{r}_{i,a} =\emptyset,$ $a \neq j-1, j,$ and  $\tilde{r}_{i,j} \cap \tilde{r}_{i,a}=\emptyset,$ $a\neq j,$
\item $K_{i,j}:=\overline{\Omega_{i,j}-(\tilde{r}_{i,j-1}\cup \tilde{\alpha}_{i,j}\cup \tilde{r}_{i,j})}$ is a compact disc and $K_{i,j} \cap \beta_{i,j}$ is a Jordan arc disjoint from the set $\{\Pgot(Q_{i,j-1}) , \Pgot(Q_{i,j})\}$ (see Figure \ref{fig:Omega}),
\item  $\|H-X\|<\epsilon_0$ on $\overline{M-\cup_{i,j}K_{i,j}},$ and

\item if $J \subset \tilde{r}_{i,j}$ is an arc connecting $\tilde{\alpha}_{i,j}\cup \tilde{\alpha}_{i,j+1}$ and $\beta_{i,j}\cup\beta_{i,j+1},$  $J_1=J \cap \Omega_{i,j}$ and $J_2=J \cap \Omega_{i,j+1},$ then 
\[
\Lcal_\c((\Acal_{i,j} \cdot H|_{J_1})_3)+\Lcal_\c((\Acal_{i,j+1} \cdot H|_{J_2})_3) > 2\rho \cdot \max\{\|\Acal_{i,j}\|,\|\Acal_{i,j+1}\|\}.
\]
\end{enumerate}
\begin{figure}[ht]
    \begin{center}
    \scalebox{0.50}{\includegraphics{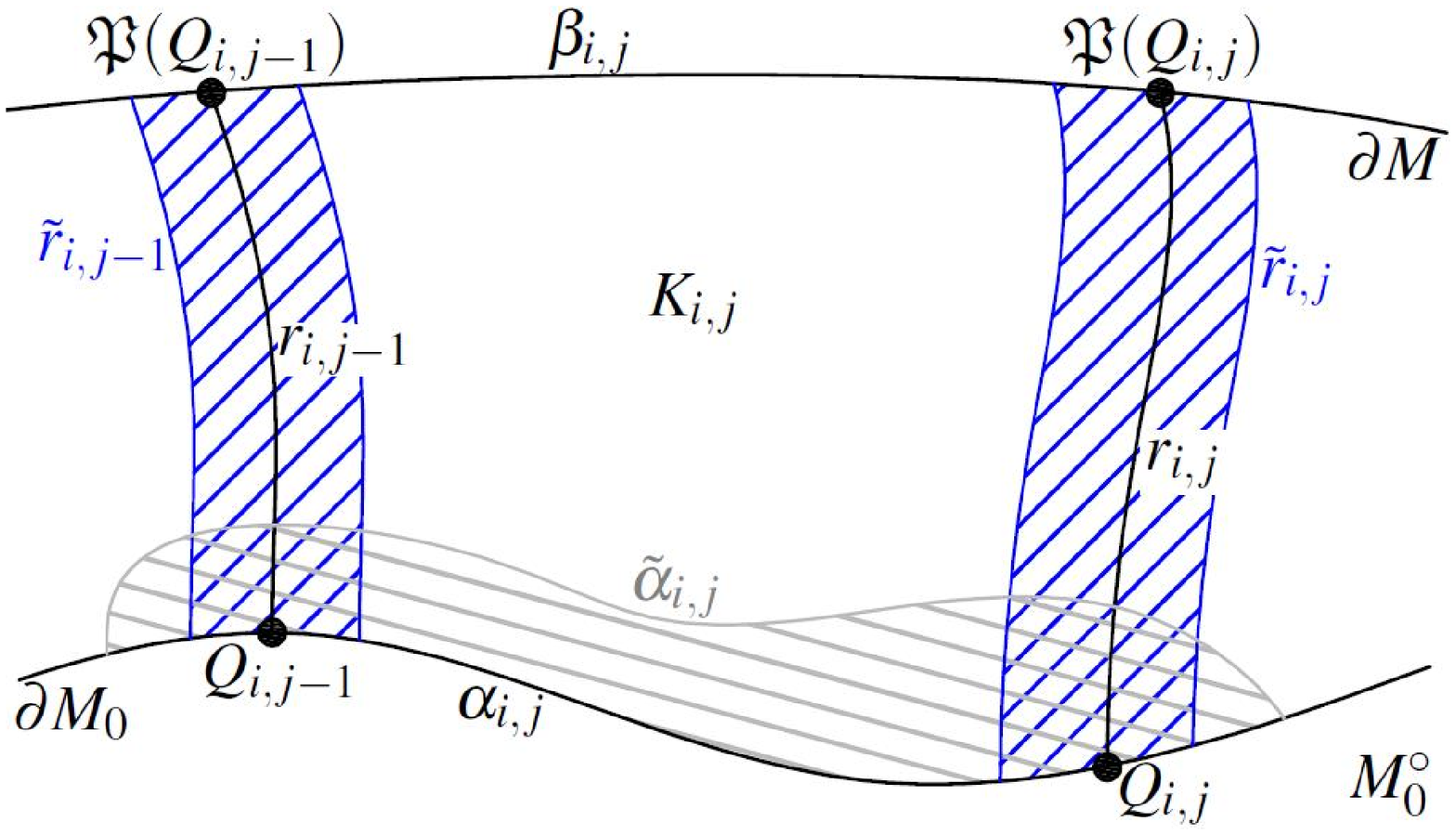}}
        \end{center}
        \vspace{-0.5cm}
\caption{$\Omega_{i,j}.$}\label{fig:Omega}
\end{figure}

The existence of such compact discs follows from a continuity argument. In order to achieve properties (E.3) and (E.4) take into account  Claim \ref{cl:antena}.

Let $\eta:\{1,\ldots,\esf\mgot\}\to\{1,\ldots,\esf\}\times \z_\mgot$ be the bijection $\eta(k)=(E(\frac{k-1}{\mgot})+1,k-1),$ where $E(\cdot)$ means integer part. 

The second deformation process in the proof of Lemma \ref{lem:CL} mainly acts in $\cup_{k}K_{\eta(k)}$ and is contained in the following claim.

\begin{claim}\label{cl:2a}
There exists a sequence $\{H_0=H, H_1,\ldots, H_{\esf\mgot}\}\subset\Nsf(M)$ satisfying the following properties:
\begin{enumerate}[\rm ({F}.1{$_k$})]
\item $\|H_k-X\|<\epsilon_0$ on $\overline{M- \big((\cup_{a=1}^k \Omega_{\eta(a)})\cup (\cup_{a=1}^{\esf\mgot}K_{\eta(a)})\big)},$

\item $\doble{H_k-H_{k-1},e_{\eta(k)}}=0,$

\item $\|H_k(Q)-X(Q)\|>2\rho+1,$ $\forall Q\in K_{\eta(a)},$ $\forall a\in\{1,\ldots, k\},$ $k\geq 1,$

\item if $J \subset \tilde{r}_{\eta(a)}$ is an arc connecting $\tilde{\alpha}_{\eta(a)}\cup \tilde{\alpha}_{\eta(a)+(0,1)}$ and $\beta_{\eta(a)}\cup\beta_{\eta(a)+(0,1)},$  $J_1 =J \cap \Omega_{\eta(a)}$ and $J_2 =J \cap \Omega_{\eta(a)+(0,1)},$ then
\[
\Lcal_\c((\Acal_{\eta(a)} \cdot H_k|_{J_1})_3)+ \Lcal_\c((\Acal_{\eta(a)+(0,1)} \cdot H_k|_{J_2})_3) > 2\rho\cdot\max\{\|\Acal_{\eta(a)}\|,\|\Acal_{\eta(a)+(0,1)}\|\},
\]
$\forall a\in \{1,\ldots, \esf\mgot\},$
\item $\|H_k-H_{k-1}\|<\epsilon_0/\esf\mgot$ on $\overline{M-\Omega_{\eta(k)}},$ and
\item $\|H_k-X\|_1<\epsilon_0$ on $M_0.$
\end{enumerate}
\end{claim}
\begin{proof}
Properties (F.1$_0$), (F.4$_0$) and (F.6$_0$) follow from (E.3), (E.4) and (D.3), whereas (F.2$_0$), (F.3$_0$) and (F.5$_0$) make no sense. We reason by induction. Assume that we have defined $H_0,\ldots,H_{k-1}$ satisfying the desired properties and let us construct $H_k.$

Label $G=\Acal_{\eta(k)}\cdot H_{k-1}\in\Nsf(M)$ and write $G=(G_1,G_2,G_3)^T.$

Let $\gamma$ denote a Jordan arc in $\tilde{\alpha}_{\eta(k)}$ disjoint from $\tilde{r}_{\eta(k)-(0,1)} \cup \tilde{r}_{\eta(k)},$ with initial point $\gamma_0\in\alpha_{\eta(k)}$ and final point $\gamma_1\in\partial K_{\eta(k)}$ and otherwise disjoint from $\partial \tilde{\alpha}_{\eta(k)}.$ Choose $\gamma$ so that $dG_3|_{\gamma}$ never vanishes. Denote $S_k=\overline{(M-\Omega_{\eta(k)})\cup \gamma\cup K_{\eta(k)}},$ and without loss of generality assume that $S_k$ is admissible. 

Let $\gamma(u),$ $u \in [0,1],$ be a smooth parameterization of $\gamma$ with $\gamma(0)=\gamma_0.$  Label  $\tau_j= \gamma([0,1/j])$ and consider the parameterization   $\tau_j(u)=\gamma(u/j),$ $u \in [0,1].$ Write $G_{3,j}(u)=G_3(\tau_j(u)),$  $u \in [0,1],$ and notice that $\frac{d G_{3,j}}{du}(0)=\frac1{j}\frac{d (G_{3}\circ \gamma)}{du}(0)$ for all $j \in \n.$

From the definition of $\Acal_{\eta(k)},$ one has $\Acal_{\eta(k)}(\doble{e_{\eta(k)}}^\bot)= \{(z_1,z_2,0) \in \c^3\,|\, z_1,\,z_2 \in \c\}.$ Set $\zeta^\pm=\lambda(1,\pm\imath,0)\in\c^3,$ where $\lambda\in\r,$ 
\begin{equation}\label{eq:lejos}
\lambda>(2\rho+2+\epsilon_0)\|\Acal_{\eta(k)}\|,
\end{equation}
and notice that $\zeta^\pm=\Acal_{\eta(k)}(\lambda( u_{\eta(k)}\pm v_{\eta(k)})).$
Set 
\[
\zeta_j=\zeta^+- \frac{(d G_{3,j}/du (0))^2}{2\curvo{\, \zeta^+,\zeta^-\,}}\; \zeta^-\in \Acal_{\eta(k)}(\doble{ e_{\eta(k)}}^\bot),\quad j \in \n,
\]
and observe that $\lim_{j \to \infty} \zeta_j =\zeta^+$ and  $\curvo{ \zeta_j,\zeta_j}=-(\frac{d G_{3,j}}{du}(0))^2$ for all $j.$ 
Define $h_j:[0,1] \to \c^3$ as 
\[
h_j(u)=G(\gamma_0)+ \imath \frac{G_{3,j}(u)-G_{3,j}(0)}{\curvo{ \,\zeta_j,\zeta_j\,}^{1/2}} \zeta_j + (0,0,G_{3,j}(u)-G_3(\gamma_0)),
\]
and notice that $\curvo{ h_j'(u),h_j'(u) }=0$ and $\doble{ h_j'(u),h_j'(u) }$ never vanishes on $[0,1],$  $j\in \n.$ Up to choosing a suitable branch of $\curvo{ \zeta_j,\zeta_j}^{1/2},$ the sequence  $\{h_j\}_{j \in \n}$ converges uniformly on $[0,1]$ to $h_\infty:[0,1] \to \c^3,$  $h_\infty(u)=u \zeta^+ + G(\gamma_0).$ From equation \eqref{eq:lejos} one has $\|h_\infty(1)-G(\gamma_0)\|=\|\zeta^+\|>(2\rho+2+\epsilon_0)\|\Acal_{\eta(k)}\|,$ hence  
there exists a large enough $j_0 \in \n$ so that $\|G(\gamma_0)-G(\gamma(1/j_0))\|=\|G(\gamma_0)-G(\tau_{j_0}(1))\|<\|\Acal_{\eta(k)}\|$ and $\|h_{j_0}(1)-G(\gamma_0)\|>(2\rho+2+\epsilon_0)\|\Acal_{\eta(k)}\|.$ In particular,
\begin{equation} \label{eq:lejos1}
\|h_{j_0}(1)-G(\tau_{j_0}(1))\|>(2\rho+1+\epsilon_0)\|\Acal_{\eta(k)}\|.
\end{equation}
 
Set $\hat{h}:\tau_{j_0} \to \c^3,$ $\hat{h}(P)=h_{j_0}(u(P)),$ where $u(P)\in [0,1]$ is the only value for which $\tau_{j_0}(u(P))=P.$   Let $\hat{G}=(\hat{G}_1,\hat{G}_2,\hat{G}_3): S_k\to \c^3$ denote the continuous map given by:   
\begin{equation}\label{eq:Ggorro}
\hat{G}|_{\overline{M-\Omega_{\eta(k)}}}=G|_{\overline{M-\Omega_{\eta(k)}}},\; \hat{G}|_{\tau_{j_0}}=\hat{h},\; \hat{G}|_{(\gamma-\tau_{j_0})\cup K_{\eta(k)}}=G|_{(\gamma-\tau_{j_0})\cup K_{\eta(k)}}-G(\tau_{j_0}(1))+\hat{h}(\tau_{j_0}(1)).
\end{equation}
Notice that $\hat{G}_3=G_3|_{S_k}.$ 
The equation $\curvo{ d\hat{G},d \hat{G}}=0$ formally holds except at the points $\gamma_0$ and $\tau_{j_0}(1)$ where smoothness could fail. 
Then, up to smooth approximation, \eqref{eq:lejos1} and \eqref{eq:Ggorro} give that  $\hat{G}$ is a generalized null curve  satisfying that 
\[
\hat{G}|_{\overline{M-\Omega_{\eta(k)}}}=G|_{\overline{M-\Omega_{\eta(k)}}},\quad \hat{G}_3=G_3|_{S_k},\quad\text{and}\quad \|(\hat{G}-G)(Q)\|>(2\rho+1+\epsilon_0)\|\Acal_{\eta(k)}\|\text{ $\forall Q\in K_{\eta(k)}$.}
\]

Fix $\epsilon_1>0$ which will be specified later. Applying Lemma \ref{lem:runge} to $S_k,$ $M$ and $\hat{G},$ we get a null curve $Z=(Z_1,Z_2,Z_3)^T\in \Nsf(M)$ such that
\begin{itemize}
\item $\|Z- \Acal_{\eta(k)}\cdot H_{k-1}\|_1<\epsilon_1$  on $\overline{M-\Omega_{\eta(k)}},$
\item $Z_3=(\Acal_{\eta(k)}\cdot H_{k-1})_3,$ and
\item $\|(Z-G)(Q)\|>(2\rho+1+\epsilon_0)\cdot\|\Acal_{\eta(k)}\|$ $\forall Q\in K_{\eta(k)}.$
\end{itemize}

Define $H_k:= \Acal_{\eta(k)}^{-1}\cdot Z\in \Nsf(M).$ From the properties of $Z$ above, $H_k$ satisfies that
\begin{enumerate}[{\rm (i)}]
\item $\|H_k- H_{k-1}\|_1< \epsilon_1 \|\Acal_{\eta(k)}^{-1}\|$ on $\overline{M-\Omega_{\eta(k)}},$
\item $\doble{ H_k- H_{k-1},e_{\eta(k)}} = 0$ (see \eqref{eq:orto} and the definition of $A_{\eta(k)}$), and
\item $\|(H_k-H_{k-1})(Q)\|>(2\rho+1+\epsilon_0)$ $\forall Q\in K_{\eta(k)}.$
\end{enumerate}

Let us check that $H_k$ is the null curve we are looking for, provided that $\epsilon_1$ is sufficiently small. Indeed, (ii) directly gives (F.2$_k$), and if $\epsilon_1$ is small enough, then (F.1$_{k-1}$), (F.6$_{k-1}$) and (i) guarantee (F.1$_k$), (F.5$_k$) and (F.6$_k$), and properties (F.1$_{k-1}$), (F.3$_{k-1}$), (i) and (iii) gives (F.3$_k$). Finally, to prove (F.4$_k$) observe that (i) gives that $\|\Acal_{\eta(a)}\cdot H_k- \Acal_{\eta(a)}\cdot H_{k-1}\|<\epsilon_1 \|\Acal_{\eta(k)}^{-1}\|\, \|\Acal_{\eta(a)}\|$ and $\|\Acal_{\eta(a)+(0,1)}\cdot H_k- \Acal_{\eta(a)+(0,1)}\cdot H_{k-1}\|<\epsilon_1 \|\Acal_{\eta(k)}^{-1}\|\, \|\Acal_{\eta(a)+(0,1)}\|$ on $\Omega_{\eta(a)}$ $\forall a\neq k,$ whereas (ii) gives that $(\Acal_{\eta(k)}\cdot H_k)_3= (\Acal_{\eta(k)}\cdot H_{k-1})_3$ on $\Omega_{\eta(k)}.$ Then  (F.4$_{k-1}$) implies (F.4$_k$) provided that $\epsilon_1$ is small enough.
\end{proof}

Consider the null curve $\hat{X}:=H_{\esf\mgot}\in\Nsf(M).$ One has

\begin{claim}\label{cl:dist1}
$\dist_{(M,\hat{X})}(M_0,\partial M)>2\rho.$ Moreover, $\dist_{(M,\hat{X})}(M_0,\cup_{a=1}^{\esf\mgot} K_{\eta(a)})>2\rho.$ 
\end{claim}
\begin{proof}
Indeed, consider a connected curve $\gamma$ in $M$ with initial point $Q\in M_0$ and final point $P\in\partial M.$ Let us distinguish cases. Assume there exists $k\in\{1,\ldots,\esf\mgot\}$ and a point $P_0\in \gamma\cap K_{\eta(k)}\neq \emptyset.$ Then there exists a point $Q_0\in\gamma\cap (r_{\eta(k)-(0,1)}\cup\alpha_{\eta(k)}\cup r_{\eta(k)})$ and one has 
\begin{eqnarray*}\label{eq:enK}
\Lcal_{(M,\hat{X})}(\gamma) =  \Lcal_{\c^3}(\hat{X}(\gamma))
& \geq & \|\hat{X}(P_0)-\hat{X}(Q_0)\|\\
 & \geq & \|(\hat{X}-X)(P_0)\|-\|X(P_0)-X(Q_0)\|-\|(\hat{X}-X)(Q_0)\|
\\
 & \geq & 2\rho+1 - \epsilon_0 -\epsilon_0 > 2\rho,
\end{eqnarray*}
where $\Lcal_{(M,\hat{X})}(\cdot)$ means length in $M$ with respect to the metric $\hat{X}^*\doble{\cdot,\cdot},$ $\Lcal_{\c^3}(\cdot)$ means Euclidean length in $\c^3.$ To achieve these bounds we have used (F.3$_{\esf\mgot}$); (A.2) and \eqref{eq:Oin}; and (F.1$_{\esf\mgot}$); respectively (assume from the beginning $\epsilon_0<1/2$). This also proves the second part of the claim.

Assume now that $\gamma\cap (\cup_{k=1}^{\esf\mgot} K_{\eta(k)})=\emptyset.$ Then there exist $k\in\{1,\ldots,\esf\mgot\}$ and a connected sub-arc $\hat{\gamma}\subset\gamma$ such that $P\in\beta_{\eta(k)}\cup\beta_{\eta(k)+(0,1)},$ $\hat{\gamma}\subset \tilde{r}_{\eta(k)}$ and $\hat{\gamma}\cap(\tilde{\alpha}_{\eta(k)}\cup\tilde{\alpha}_{\eta(k)+(0,1)})\neq \emptyset.$ Remark \ref{rem:AF} and (F.4$_{\esf\mgot}$) give $\Lcal_{(M,\hat{X})}(\gamma)\geq \Lcal_{\c^3}(\hat{X}(\gamma))\geq  \Lcal_{\c^3}(\hat{X}(\hat{\gamma}))>2\rho$ and the proof is done.
\end{proof}

\begin{claim}\label{cl:dist2}
$\|\hat{X}(Q)-\Fgot(\Pgot(Q))\|<\sqrt{4\rho^2+\mu^2}+\epsilon$ for any $Q\in M-M_0$ with $\dist_{(M,\hat{X})}(M_0,Q)<2\rho.$
\end{claim}
\begin{proof}
Obviously there exist $k\in\{1,\ldots,\esf\mgot\}$ and $P\in r_{\eta(k)-(0,1)}\cup\alpha_{\eta(k)}\cup r_{\eta(k)}$ such that
\begin{equation}\label{eq:2rho}
Q\in\Omega_{\eta(k)}-K_{\eta(k)} \quad\text{and}\quad \dist_{(M,\hat{X})}(P,Q)<2\rho,
\end{equation}
see Claim \ref{cl:dist1}. By Pitagoras Theorem,
\begin{equation}\label{eq:pita1}
\|\hat{X}(Q)-\Fgot(\Pgot(Q))\|= \sqrt{|\doble{\hat{X}(Q)-\Fgot(\Pgot(Q)),e_{\eta(k)}}|^2+\|\Pi_{\eta(k)}(\hat{X}(Q)-\Fgot(\Pgot(Q)))\|^2}.
\end{equation}
Let us bound the two addens in the right part of this inequality. On the one hand,
\begin{equation}\label{eq:pita2}
\begin{array}{rcl}
|\doble{\hat{X}(Q)-\Fgot(\Pgot(Q)),e_{\eta(k)}}| & \leq & \|\hat{X}(Q)-H_k(Q)\|+|\doble{(H_k-H_{k-1})(Q),e_{\eta(k)}}| 
\\
& + & \|H_{k-1}(Q)-X(Q)\| + \|X(Q)-\Fgot(\Pgot(Q))\|
\\
& < & \epsilon_0 + 0 + \epsilon_0 + \mu,
\end{array}
\end{equation}
where we have used (F.5$_a$), $a=k+1,\ldots,\esf\mgot;$ (F.2$_k$); (F.1$_{k-1}$); and (A.3). On the other hand,
\begin{equation}\label{eq:pita3}
\begin{array}{rcl}
\|\Pi_{\eta(k)}(\hat{X}(Q)-\Fgot(\Pgot(Q)))\| & \leq & \|\hat{X}(Q)-\hat{X}(P)\|+\|\hat{X}(P)-X(P)\|+\|X(P)-X(P_{\eta(k)})\|
\\
& + & \|\Pi_{\eta(k)}(X(P_{\eta(k)})-\Fgot(\Pgot(P_{\eta(k)})))\|+\|\Fgot(\Pgot(P_{\eta(k)}))-\Fgot(\Pgot(Q))\|
\\
& < & 2\rho+\epsilon_0+\epsilon_0+ \epsilon_0 + \epsilon_0,
\end{array}
\end{equation}
where we have used \eqref{eq:2rho}; (F.1$_{\esf\mgot}$); (A.2); \eqref{eq:eij}; and (A.2); to get the corresponding bounds. Combining \eqref{eq:pita1}, \eqref{eq:pita2} and \eqref{eq:pita3} one has
\[
\|\hat{X}(Q)-\Fgot(\Pgot(Q))\|<\sqrt{4\rho^2+\mu^2+(4\mu + 16\rho)\epsilon_0+20 \epsilon_0}<\sqrt{4\rho^2+\mu^2}+\epsilon,
\]
provided that $\epsilon_0$ is chosen from the beginning to satisfy the second inequality, and we are done. 
\end{proof}

By Claim \ref{cl:dist1} there exists $\Mcal\in \Bsf(\Rcal)$ such that
\begin{enumerate}
\item[{\rm (G.1)}] $M_0<\Mcal<M,$ and
\item[{\rm (G.2)}] $\rho<\dist_{(\Mcal,\hat{X})}(M_0,Q)<2\rho$ $\forall Q\in \partial\Mcal.$
\end{enumerate}
Furthermore, up to infinitesimal variations of the boundary curves of $\Mcal,$ we can guarantee that 
\begin{enumerate}
\item[{\rm (G.3)}] $\hat{X}|_{\partial\Mcal}:\partial\Mcal\to \c^3$ is an embedding.
\end{enumerate}

Set $\Xcal:=\hat{X}|_\Mcal\in\Nsf(\Mcal).$ Item (\ref{lem:CL}.a) follows from \eqref{eq:M_0} and (G.1). (G.3) directly gives(\ref{lem:CL}.b). (\ref{lem:CL}.c) is implied by (G.2). (\ref{lem:CL}.d) follows from (F.6$_{\esf\mgot}$) and \eqref{eq:M_0} provided that $\epsilon_0$ is chosen less than $\epsilon.$ Finally, Claim \ref{cl:dist2} and (G.2) give (\ref{lem:CL}.e).
This proves Lemma \ref{lem:CL}. 


\section{Main Lemma}\label{sec:ML}

In this section we show the following density type result for null curves:

\begin{lemma}\label{lem:ML}
Let $V,W\in\Bsf(\Rcal)$ with $V<W.$ Let $Y\in\Nsf(W)$ and let $\lambda>0.$

Then there exist $\Wcal\in\Bsf(\Rcal)$ and $\Ycal\in\Nsf(\Wcal)$ satisfying that
\begin{enumerate}[{\rm (\ref{lem:ML}.a)}]
\item $\Ycal|_{\partial \Wcal}:\partial\Wcal\to\c^3$ is an embedding,
\item $V<\Wcal<W,$
\item $\lambda<\dist_{(\Wcal,\Ycal)}(V,\partial \Wcal),$
\item $\|\Ycal-Y\|<1/\lambda$ on $\Wcal,$
\item $\delta^H(\Ycal(\partial\Wcal),Y(\partial W))<1/\lambda,$ and
\item $\|\Ycal-Y\|_1<1/\lambda$ on $V.$
\end{enumerate}
\end{lemma}

The most important points in this lemma are Items (\ref{lem:ML}.c) and (\ref{lem:ML}.d). They assert that the immersion $\Ycal$ is very close to $Y$ in whole its domain of definition and, however, its intrinsic radious is much larger. Therefore, as we will see in Theorem \ref{th:fun} below, Lemma \ref{lem:ML} is a powerful tool for constructing complete bounded null curves in $\c^3.$ 

\subsection{Proof of Lemma \ref{lem:ML}}

Fix a positive constant $\rho_1$ which will be specified later. Consider sequences $\{\rho_n\}_{n\in\n}$ and $\{\mu_n\}_{n\in\n}$ given by
\begin{equation}\label{eq:pitag}
\rho_n=\rho_1+\sum_{j=2}^n \frac{\csf}{j}\quad\text{and}\quad \mu_n=\sqrt{\mu_{n-1}^2+4\big(\frac{\csf}{n}\big)^2}+\frac{\csf}{n^2},\quad \forall n\geq 2,
\end{equation}
where $\csf$ and $\mu_1$ are small enough positive constants so that
\begin{equation}\label{eq:csf}
\mu_n <\frac1{2\lambda}\quad\forall n\in\n.
\end{equation}

For convenience write $W_0$ for $W.$ Let $\Tsf_0$ be a metric tubular neighborhood of $\partial W_0$ in $\Rcal$ disjoint from $\overline{V}$  
and denote by $\Pgot_0:\Tsf_0\to\partial W_0$ the natural projection.

\begin{claim}\label{cl:ML}
There exists a sequence $\{\Xi_n=(W_n,\Tsf_n,Y_n)\}_{n\in\n},$ where $W_n\in\Bsf(\Rcal),$ $\Tsf_n$ is a metric tubular neighborhood of $\partial W_n$ in $\Rcal$ and $Y_n\in\Nsf(W_n),$  satisfying the following properties:
\begin{enumerate}[{\rm (1$_n$)}]
\item $Y_n|_{\partial W_n}:\partial W_n\to \c^3$ is an embedding,
\item  $V<W_n<W_{n-1}<W_0$ and $\Tsf_n \subset \Tsf_{n-1}\subset \Tsf_0,$   $n\geq 1,$  
\item $\rho_n<\dist_{(W_n,Y_n)}(V,\partial W_n),$
\item $\max\{\|Y-Y\circ\Pgot_0\circ \ldots \circ \Pgot_{n-1} \|\,,\,\|Y_n-Y\circ\Pgot_0\circ \ldots \circ \Pgot_{n-1} \|\}<\mu_n$ on $\partial W_n,$ where $\Pgot_j:\Tsf_j\to\partial W_j$ is the corresponding orthogonal projection for all $j,$   and
\item $\|Y_n-Y\|_1<1/\lambda$ on $V.$
\end{enumerate}
\end{claim}

\begin{proof}
Let us follow an inductive method. Take $W_1\in\Bsf(\Rcal),$ $V<W_1<W_0,$ close enough to $W_0$ so that $\partial W_1\subset\Tsf_0$ and 
\begin{equation}\label{eq:W_1}
\|Y-Y\circ\Pgot_0\|<\mu_1\quad \text{on }\overline{W_0-W_1}.
\end{equation}
For the basis of the induction choose $W_1$ and $Y_1=Y|_{W_1}.$ 
This gives (4$_1$). (5$_1$) trivially holds. Assume that $\rho_1$ is small enough so that (3$_1$) is satisfied. Finally, up to an infinitesimal variation of $\partial W_1$ we can guarantee that (1$_1$) holds as well.  Then choose $\Tsf_1$ any metric tubular neighborhood of $\partial W_1$ such that $\Tsf_1 \subset \Tsf_0$  and set $\Xi_1=(W_1,\Tsf_1,Y_1).$

For the inductive step, assume that $\Xi_1,\ldots, \Xi_{n-1}$ have been already constructed satisfying the desired properties, $n\geq 2.$ Take $W'_{n}\in\Bsf(\Rcal),$ $V<W'_{n}<W_{n-1},$ close enough to $W_{n-1}$ so that $\partial W'_{n}\subset\Tsf_{n-1}$ and
\begin{equation}\label{eq:W_n}
\|Y-Y\circ\Pgot_0\circ \ldots \circ \Pgot_{n-1}\|<\mu_n \quad \text{on}\quad \overline{W_{n-1}-W'_{n}}.
\end{equation}
Here we have taken into account that $\Pgot_{n-1}|_{\partial W_{n-1}}$ is the identity map, (4$_{n-1}$) and the inequality $\mu_{n-1}<\mu_n.$ 
Then set $(W_n,Y_n)$ as the couple $(\Mcal,\Xcal)$ obtained by Lemma \ref{lem:CL} applied to the data
\[
M=W_{n-1},\quad \Tsf_{n-1},\quad \Pgot=\Pgot_{n-1},\quad \mu=\mu_{n-1},\quad X=Y_{n-1},\]
\[\Fgot=(Y\circ\Pgot_0\circ \ldots \circ \Pgot_{n-2})|_{\partial W_{n-1}},\quad  N=W'_{n},\quad \rho=\frac{\csf}{n}\quad\text{and}\quad \epsilon<\frac{\csf}{n^2},
\]
and choose $\Tsf_n$ any metric tubular neighborhood of $\partial W_n$ such that $\Tsf_{n} \subset \Tsf_{n-1}.$
Then (\ref{lem:CL}.b), (\ref{lem:CL}.a) and (\ref{lem:CL}.c)  give (1$_n$), (2$_n$) and (3$_n$), respectively. (4$_n$) follows from (\ref{lem:CL}.e) and \eqref{eq:W_n} (see also \eqref{eq:pitag}), and finally (\ref{lem:CL}.d) and (5$_{n-1}$) give (5$_n$) for $\epsilon$ small enough.
\end{proof}

Since $\{\rho_n\}_{n\in\n}\nearrow +\infty$ then there exists $k\in\n$ such that $\rho_k>\lambda.$ Define $\Wcal=W_k$ and $\Ycal=Y_k.$ Properties (\ref{lem:ML}.a), (\ref{lem:ML}.b), (\ref{lem:ML}.c) and (\ref{lem:ML}.f) trivially follow from (1$_k$), (2$_k$), (3$_k$) and (5$_k$), respectively. By (4$_k$)  and \eqref{eq:csf} one has
\[
\|\Ycal-Y\|\leq \|Y_k-Y\circ\Pgot_0\circ \ldots \circ \Pgot_{k-1}\|+\|Y\circ\Pgot_0\circ \ldots \circ \Pgot_{k-1}-Y\|<\mu_k+\mu_k<\frac1{\lambda}\quad \text{on }\partial \Wcal.
\]
Then, by the Maximum Principle for harmonic maps $\|\Ycal-Y\|<1/\lambda$ on $\Wcal$ proving (\ref{lem:ML}.d). Finally, (\ref{lem:ML}.e) follows from the same argument.
The proof of Lemma \ref{lem:ML} is done.


\section{Main Theorem}\label{sec:th}

We are now ready to state and prove the main result of this paper. Observe that Main Theorem in the introduction directly follows from the following 

\begin{theorem}\label{th:fun}
Let $M,N\in\Bsf(\Rcal),$ $N<M.$ Let $X\in\Nsf(M)$ and let $\epsilon>0.$

Then there exist a relatively compact domain $\Dcal\subset \Rcal$ and a compact complete null curve $\Xcal:\overline{\Dcal}\to\c^3$ satisfying that
\begin{enumerate}[{\rm (\ref{th:fun}.a)}]
\item $\Xcal|_{Fr(\Dcal)}:Fr(\Dcal)\to\c^3$ is an embedding,
\item $N\subset\Dcal\subset\overline{\Dcal}\subset M^\circ,$ $\Dcal$ is homeomorphic to $\Rcal$ and both $\Dcal$ and $\overline{\Dcal}$ are Runge,
\item $\|\Xcal-X\|_1<\epsilon$ on $N$ and $\|\Xcal-X\|<\epsilon$ on $\overline{\Dcal},$
\item $\delta^H(\Xcal(Fr(\Dcal)),X(\partial M))<\epsilon,$ and
\item the Hausdorff dimension of $\Xcal(Fr(\Dcal))$ is  $1.$
\end{enumerate}
\end{theorem}

Through the proof of Theorem \ref{th:fun} we will need the following notation. For any $k\in\n,$ any compact set $K\subset\Rcal$ and any continuous injective map $f:K\to\c^3,$ set
\[
\Psi(K,f,k)=\frac{1}{2k^2}\cdot\inf\big\{\|f(P)-f(Q)\|\;\big|\; P,Q\in K,\; \dist_{(\Rcal,\omega)}(P,Q)>\frac1{k}\big\},
\]
where $\dist_{(\Rcal,\omega)}(\cdot,\cdot)$ means intrinsic distance with respect to the Riemannian metric $\omega$ (see Remark \ref{re:prime}). 
Notice that $\Psi(K,f,k)>0.$ The operator $\Psi$ will be useful in order to guarantee that $\Xcal|_{Fr(\Dcal)}:Fr(\Dcal)\to\c^3$ is injective.

\subsection{Proof of Theorem \ref{th:fun}}

Let $\csf$ be a positive constant which will be specified later. Let $\epsilon_1>0$ and let $M_1\in\Bsf(\Rcal)$ satisfying that
\begin{enumerate}[{\rm (i)}]
\item $N<M_1<M,$ 
\item $X|_{\partial M_1}:\partial M_1\to\c^3$ is an embedding,
\item $\delta^H(X(\partial M_1),X(\partial M))<\epsilon_1,$ and
\item $\epsilon_1<\csf/2.$ 
\end{enumerate}
Such $M_1$ exists by a continuity argument (Items (i) and (iii)). For Item (ii), make an infinitesimal variation of $\partial M_1$ if necessary. Set $X_1:=X|_{M_1}$ and let us prove the following

\begin{claim}\label{cl:th}
There exists a sequence $\{\Theta_n=(M_n,X_n,\Tsf_n,\epsilon_n,\tau_n)\}_{n\in\n},$ where $M_n\in\Bsf(\Rcal),$ $X_n\in\Nsf(M_n),$ $\Tsf_n$ is a metric tubular neighborhood of $\partial M_n$ in $M_n,$  $0<\epsilon_n<\csf/2^n,$ and $\tau_n> 0,$ satisfying that
\begin{enumerate}[{\rm (1$_n$)}]
\item $X_n|_{\overline{\Tsf}_n}:\overline{\Tsf}_n\to\c^3$ is an embedding,
\item $N<M_{n-1}-\Tsf_{n-1}<M_n-\Tsf_n<M_n<M_{n-1}<M,$ $n\geq 2,$
\item $\|X_n-X_{n-1}\|_1<\epsilon_n$ on $M_{n-1}-\Tsf_{n-1}$ and $\|X_n-X_{n-1}\|<\epsilon_n$ on $M_n,$ $n\geq 2,$
\item $\dist_{(M_n,X_n)}(N,\Tsf_n)>1/\epsilon_n,$ $n\geq 2,$
\item $\delta^H(X_n(\partial M_n),X_{n-1}(\partial M_{n-1}))<\epsilon_n,$ $n\geq 2,$
\item there exist $a_n:=\esf\cdot E((\tau_n)^{n+1})$ points $x_{n,1},\ldots, x_{n,a_n}$ in $X_n(\partial M_n)\subset\c^3$ such that
\[
\dist_{\c^3}(X_n(\overline{\Tsf}_n),\{x_{n,1},\ldots, x_{n,a_n}\})<\big(1/{\tau_n}\big)^n,
\]
where $E(\cdot)$ and $\dist_{\c^3}(\cdot)$ means integer part and Euclidean distance in $\c^3,$ respectively,
\item $\displaystyle\epsilon_n<\min\left\{\epsilon_{n-1}, \varrho_{n-1},\frac1{n^2(\tau_{n-1})^n}\,,\,\Psi\big( \overline{\Tsf}_{n-1}\,,\,X_{n-1}|_{\overline{\Tsf}_{n-1}}\,,\, n \big)\right\},$ where $$\varrho_{n-1}=2^{-n} \min \left\{  \min_{M_{k-1}-\Tsf_{k-1}} \big\| \frac{d X_k}{\omega}\big\|\,|\; k=1,\ldots,n-1 \right\}>0,\; \;n\geq 2,$$
\item $\tau_n\geq \tau_{n-1}+1\geq n,$ $n\geq 2,$ and
\item $\max\{\dist_{(\Rcal,\omega)}(P,\partial M_n)\;|\; P\in \overline{\Tsf}_n\}<\epsilon_n.$
\end{enumerate}
\end{claim}

\begin{proof}
Let us follow an inductive process. We have already introduced $M_1,$ $X_1$ and $\epsilon_1,$ hence for setting $\Theta_1$ we only must define $\Tsf(M_1)$ and $\tau_1$ satisfying the corresponding properties. 

Set $\tau_1=\max\{1,\Lcal_{\c^3}(X_1(\partial M_1))\},$ where $\Lcal_{\c^3}(\cdot)$ means Euclidean length of curves in $\c^3.$

Write $\gamma_1,\ldots,\gamma_\esf$ for the components of $\partial M_1.$ For any $i=1\ldots,\esf,$ let $y_{i,1},\ldots,y_{i,a_1/\esf}$ be points on $X_1(\gamma_i)$ with mutual distance along $X_1(\gamma_i)$ and denote by $b_{1,i}$ this distance. Set $b_1=\max\{b_{1,1},\ldots,b_{1,\esf}\}$ and $\{x_{1,1},\ldots, x_{1,a_1}\}=\{y_{i,j}\,|\, i=1,\ldots,\esf,\, j=1,\ldots,a_1/\esf\}.$  Then $b_1<1/\tau_1,$ hence 
\begin{equation}\label{eq:puntosH}
\dist_{\c^3}(X_1(\partial M_1)),\{x_{1,1},\ldots, x_{1,a_1}\})<1/\tau_1.
\end{equation}
By (i), (ii), \eqref{eq:puntosH} and a continuity argument, there exists a tubular neighborhood $\Tsf(M_1)$ of $\partial M_1$ on $M_1$ such that (1$_1$), (6$_1$), (9$_1$) and $N<M_1-\Tsf(M_1)$ (which is the meaningful part of  (2$_1$)) hold. The remaining properties of $\Theta_1$ make no sense.

For the inductive step, assume that we have stated $\Theta_1,\ldots,\Theta_{n-1},$ $n\geq 2,$ satisfying the corresponding properties and let us construct $\Theta_n.$

Observe that (1$_{n-1}$) implies that $\Psi(\overline{\Tsf}_{n-1},X_{n-1}|_{\overline{\Tsf}_{n-1}},n)>0.$ Then we can choose $\epsilon_n<\csf/2^n$ satisfying (7$_n$). Define $M_n=\Wcal\in\Bsf(\Rcal)$ and $X_n=\Ycal\in\Nsf(\Wcal)$ given by Lemma \ref{lem:ML} applied to the data
\[
V=M_{n-1}-\Tsf_{n-1},\quad W=M_{n-1},\quad Y=X_{n-1}\quad\text{and}\quad \lambda=\frac1{\epsilon_n}.
\]

Then one has
\begin{enumerate}[{\rm (a)}]
\item $X_n|_{\partial M_n}:\partial M_n\to\c^3$ is an embedding,
\item $M_{n-1}-\Tsf_{n-1}<M_n<M_{n-1},$
\item $\dist_{(M_n,X_n)}(M_{n-1}-\Tsf_{n-1},\partial M_n)>1/\epsilon_n,$
\item $\|X_n-X_{n-1}\|_1<\epsilon_n$ on $M_{n-1}-\Tsf_{n-1}$ and $\|X_n-X_{n-1}\|<\epsilon_n$ on $M_n,$ and
\item $\delta^H(X_n(\partial M_n),X_{n-1}(\partial M_{n-1}))<\epsilon_n.$
\end{enumerate}

Items (d) and (e) directly give (3$_n$) and (5$_n$). Set $\tau_n=\max\{\tau_{n-1}+1,\Lcal_{\c^3}(X_n(\partial M_n))\}.$ Hence (8$_n$) holds. Write $\xi_1,\ldots,\xi_\esf$ for the components of $\partial M_n.$ Recall that $X_n(\xi_i)$ is a Jordan curve by (a), $i=1,\ldots,\esf.$ For any $i=1,\ldots,\esf$ consider $z_{i,1},\ldots,z_{i,a_n/\esf}$ points in $X_n(\xi_i)$ with mutual distance along $X_n(\xi_i)$ and denote by $b_{n,i}$ this distance. Set $b_n=\max\{b_{n,i}\}_{i=1}^\esf$ and $\{x_{n,1},\ldots,x_{n,a_n}\}=\{z_{i,j}\,|\,i=1,\ldots,\esf,\, j=1,\ldots,a_n/\esf\}.$ Then $b_n<1/(\tau_n)^n,$ hence
\begin{equation}\label{eq:tau_n}
\dist_{\c^3}(X_n(\partial M_n),\{x_{n,1},\ldots,x_{n,a_n}\})<\big(1/{\tau_n}\big)^n.
\end{equation}

Take a metric tubular neighborhood $\Tsf_n$ of $\partial M_n$ in $M_n.$  By a continuity argument and choosing $\Tsf_n$ small enough, all the properties from  (1$_n$) to (9$_n$) hold. Indeed, (9$_n$) can be trivially guaranteed, and for (1$_n$); (2$_n$); (4$_n$); and (6$_n$), take into account  (a); (2$_{n-1}$) and (b); (c); and \eqref{eq:tau_n}, respectively.
This proves the claim.
\end{proof}

Set $\Ncal_n=(M_n-\Tsf_n)^\circ$ and $\Scal_n=\Rcal-M_n$ for all $n\in \n.$ Define
\[
\Dcal:=\bigcup_{n\in \n} \Ncal_n   
\]
and notice that, by (2$_n$) and (9$_n$), $n\in\n,$ one has
\[
\overline{\Dcal}=\bigcap_{n\in\n} M_n.
\]
By (2$_n$), $n\in\n,$ and elementary topological arguments one has that $\Dcal$ is Runge and homeomorphic to $\Rcal.$ On the other hand $\Scal_n$ consists of $\esf$ pairwise disjoint open annuli and $\Scal_n\subset\Scal_{n+1}\subset\Rcal-N$ $\forall n\in\n,$ hence $\Rcal-\overline{\Dcal}=\cup_{n\in\n} \Scal_n$ consists of $\esf$ pairwise disjoint open annuli as well. This implies that $\overline{\Dcal}$ is a Runge compact set. Then Item (\ref{th:fun}.b) follows from (2$_n$), $n\in\n.$  

By (2$_n$), (3$_n$) and (7$_n$), $n\in\n,$ $\{X_n|_{\overline{\Dcal}}\}_{n\in\n}$ uniformly converges to a continuous map 
\[
\Xcal:\overline{\Dcal}\to\c^3
\]
with $\|\Xcal-X\|_1<\csf$ on $N$ and $\|\Xcal-X\|<\csf$ on $\overline{\Dcal}.$ This, (iii) and (iv) give (\ref{th:fun}.c) and (\ref{th:fun}.d) provided that $\csf<\epsilon$ from the beginning. 
By (3$_n$), $n\in\n,$ $\Xcal|_{\Dcal}$ is a holomorphic map and $\curvo{d\Xcal,d\Xcal}=0$ on $\Dcal.$ To check that $\doble{d\Xcal,d\Xcal}$ never vanishes on $\Dcal$ argue as follows. Fix $P\in\Dcal$ and take $n_0\in\n$ such that $P\in M_{n-1}-\Tsf_{n-1}$ $\forall n\geq n_0.$ Then
\begin{eqnarray*}
\big\|\frac{d\Xcal}{\omega}\big\|(P) & \geq & \big\|\frac{dX_{n_0}}{\omega}\big\|(P)-\sum_{n>n_0}\|X_n-X_{n-1}\|_1\\
& \geq & \big\|\frac{dX_{n_0}}{\omega}\big\|(P)-\sum_{n>n_0}\epsilon_n
\\
 & \geq & \big\|\frac{dX_{n_0}}{\omega}\big\|(P)-\sum_{n>n_0}\varrho_{n-1}\\
 & \geq & \big( 1- \sum_{n>n_0}2^{-n}\big)\big\|\frac{dX_{n_0}}{\omega}\big\|(P),
 \end{eqnarray*}
where we have used (3$_n$) and (7$_n$). Therefore
\begin{eqnarray}\label{eq:ultima}
\big\|\frac{d\Xcal}{\omega}\big\|(P) & \geq & \frac12 \big\|\frac{dX_{n_0}}{\omega}\big\|(P)>0,
\end{eqnarray}
In particular $\doble{d\Xcal,d\Xcal}(p)\neq 0$ and $\Xcal|_{\Dcal}$ is an immersion. This proves that $\Xcal|_{\Dcal}$ is a null curve.

From (4$_n$), $n\in\n,$ and \eqref{eq:ultima}, one has
\[
\dist_{(\Dcal,\Xcal)}(N,Fr(\Dcal))=\lim_{n\to\infty} \dist_{(\Dcal,\Xcal)}(N,\Tsf_n)
\geq \frac12 \lim_{n\to\infty} \dist_{(M_n,X_n)}(N,\Tsf_n)>\frac12\lim_{n\to\infty}\frac1{\epsilon_n}=\infty,
\]
hence $\Xcal|_{\Dcal}$ is complete and $\Xcal:\overline{\Dcal}\to\c^3$ is a compact complete null curve as claimed. 

Let us show that $\Xcal|_{Fr(\Dcal)}:Fr(\Dcal)\to\c^3$ is injective. Indeed, take points $P$ and $Q$ in $Fr(\Dcal)\subset \Tsf_n$ (see (2$_n$)), $n\in\n,$ $P\neq Q.$ Take $n_0\in\n$ large enough so that $\dist_{(\Rcal,\omega)}(P,Q)>1/n_0.$ Then, for any $n>n_0,$
\begin{eqnarray*}
\|X_{n-1}(P)-X_{n-1}(Q)\| & \leq & \|X_{n-1}(P)-X_{n}(P)\| + \|X_{n}(P)-X_{n}(Q)\| + \|X_{n}(Q)-X_{n-1}(Q)\|
\\
 & < & \epsilon_n + \|X_{n}(P)-X_{n}(Q)\| + \epsilon_n
\\
 & < & \frac{1}{n^2} \|X_{n-1}(P)-X_{n-1}(Q)\| + \|X_{n}(P)-X_{n}(Q)\|,
\end{eqnarray*}
where we have used (3$_n$), (7$_n$) and the definition of $\Psi.$ Then 
\[
\|X_{n_0+k}(P)-X_{n_0+k}(Q)\|>\big(\prod_{n=n_0+1}^{n_0+k}\big(1-\frac1{n^2}\big)\big)\|X_{n_0}(P)-X_{n_0}(Q)\|\quad\forall k\in\n.
\]
Taking limits in this expression as $k$ goes to infinity, $\|\Xcal(P)-\Xcal(Q)\|>\frac12\|X_{n_0}(P)-X_{n_0}(Q)\|>0$ and (\ref{th:fun}.a) holds.

Finally let us check (\ref{th:fun}.e). Take $n\in\n$ with $\sum_{m=n}^\infty 1/m^2<1.$ Then, for any $P\in Fr(\Dcal)$ and any $k>n,$ properties (7$_m$) and (9$_m$), $m=n+1,\ldots, k,$ give
\begin{eqnarray*}
\|X_k(P)-X_{n}(P)\| & \leq & \|X_k(P)-X_{k-1}(P)\|+\ldots+\|X_{n+1}(P)-X_{n}(P)\|
\\
 & < & \frac1{k^2(\tau_{k-1})^k}+\ldots+\frac1{(n+1)^2(\tau_{n})^{n+1}}
 \\
 &<&\big(\sum_{m=n+1}^k \frac1{m^2}\big)\frac1{(\tau_{n})^{n+1}}<\frac1{(\tau_{n})^{n+1}}<\frac1{(\tau_{n})^{n}}.
\end{eqnarray*}
Taking limits in the above inequality as $k$ goes to infinity one has $\|\Xcal(P)-X_n(P)\|<1/{(\tau_{n})^{n}}$ for any large enough $n\in\n.$ Combining this inequality and properties (6$_n$) one obtains that $\dist(\Xcal(Fr(\Dcal)),\{x_{n,1},\ldots,x_{n,a_n}\})<2/{(\tau_n)^n}.$ Since $a_n\cdot(2/{(\tau_n)^n})^{1+1/n}<4\esf$ then the Hausdorff measure $\Hcal^1(\Xcal(Fr(\Dcal)))<\infty,$ and so the Hausdorff dimension of $\Xcal(Fr(\Dcal))$ is at most $1.$ On the other hand, if $\csf$ is taken small enough from the beginning, then each of the $\esf$ connected components of $\Xcal(Fr(\Dcal))$ contains more than one point (see Item (\ref{th:fun}.c)). Therefore the  Hausdorff dimension of $\Xcal(Fr(\Dcal))$ is at least $1,$ hence equals to $1.$ 
This proves  (\ref{th:fun}.e) and the theorem. 


\subsection{Applications to other ambient manifolds}\label{sec:coros}

Given an open Riemann surface $\Ncal,$ a map $\varphi:\Ncal\to{\rm SL}(2,\c)$ is said to be a null curve if and only if $\varphi$ is a holomorphic immersion and $\det(d\varphi)=0.$ 
The hyperbolic 3-space $\H^3$ can be identified to ${\rm SL}(2,\c)/{\rm SU}(2),$ and Bryant's projection 
\[
\Bcal:{\rm SL}(2,\c)\to \H^3={\rm SL}(2,\c)/{\rm SU}(2),\quad \Bcal(A)=A\cdot \overline{A}^T,
\]
maps complete null curves in ${\rm SL}(2,\c)$ into conformal complete immersions of constant mean curvature $1$ in $\H^3.$ 

Our Main Theorem directly implies  similar existence results for complex curves in $\c^2,$  minimal surfaces in $\r^3,$ null curves in ${\rm SL}(2,\c)$ and Bryant surfaces in $\h^3$ of arbitrary finite topology (see \cite{MUY1,MUY2,AL2} and the references therein for more details). All of them  have been compiled in the following corollary.

\begin{corollary}\label{co:consecuencias}
Let $S$ be an open orientable surface of finite topology and $\Mgot\in\{\c^2,\r^3,{\rm SL}(2,\c),\h^3\}.$

Then there exist an open Riemann surface $\Mcal,$ a relatively compact domain $\Dcal\subset\Mcal$ such that $\Dcal$ is homeomorphic to $S$ and both $\Dcal$ and $\overline{\Dcal}$ are Runge subsets in $\Mcal,$ and a continuous map $X:\overline{\Dcal}\to\Mgot$ satisfying that $X|_{Fr(\Dcal)}:Fr(\Dcal)\to\Mgot$ is an embedding, the Hausdorff dimension of $X(Fr(\Dcal))$ is  $1$ and either of the following statements:
\begin{enumerate}[{\rm (i)}]
\item $\Mgot=\c^2$ and $X|_\Dcal$ is a complete holomorphic immersion. Furthermore, if $\epsilon>0$ and $\Sigma$ is a finite family of closed curves in $\c^2$ which is spanned by a compact complex curve in $\c^2$ which is homeomorphic to $S$ and can be lifted to a null curve in $\c^3,$ then $X$ and $\Dcal$ can be chosen so that $\delta_{\c^2}^H(\Sigma,X(Fr(\Dcal)))<\epsilon,$ where $\delta_{\c^2}^H(\cdot,\cdot)$ means Hausdorff distance in $\c^2.$

\item $\Mgot=\r^3,$ $X|_\Dcal$ is a conformal complete minimal immersion and the conjugate immersion $(X|_\Dcal)^*$ of $X|_\Dcal$ is well defined, extends continuosly to $\overline{\Dcal}$ and satisfies the same properties. Furthermore, if $\epsilon>0$ and $\Sigma,$ $\Sigma^*$ are two finite families of closed curves in $\r^3$ spanned by a minimal surface homeomorphic to $S$ and its (well defined) conjugate surface, respectively, then $X$ and $\Dcal$ can be chosen so that $\max\{\delta_{\r^3}^H(\Sigma,X(Fr(\Dcal))),\delta_{\r^3}^H(\Sigma^*,X^*(Fr(\Dcal)))\}<\epsilon,$ where $\delta_{\r^3}^H(\cdot,\cdot)$ means Hausdorff distance in $\r^3.$

\item $\Mgot={\rm SL}(2,\c)$ and $X|_\Dcal$ is a complete null holomorphic curve. Furthermore, if $\epsilon>0$ and $\Sigma$ is a finite family of closed curves in ${\rm SL}(2,\c)$ spanned by a null holomorphic curve homeomorphic to $S,$ then $X$ and $\Dcal$ can be chosen so that $\delta_{{\rm SL}(2,\c)}^H(\Sigma,X(Fr(\Dcal)))<\epsilon,$ where $\delta_{{\rm SL}(2,\c)}^H(\cdot,\cdot)$ means Hausdorff distance in ${\rm SL}(2,\c).$

\item  $\Mgot=\h^3$ and $X|_\Dcal$ is a conformal complete CMC-1 surface in $\h^3.$ Furthermore, if $\epsilon>0$ and $\Sigma$ is a finite family of closed curves in $\h^3$ spanned by a CMC-1 surface which is homeomorphic to $S$ and can be lifted to a null holomorphic curve in ${\rm SL}(2,\c),$ then $X$ and $\Dcal$ can be chosen so that $\delta_{\H^3}^H(\Sigma,X(Fr(\Dcal)))<\epsilon,$ where $\delta_{\H^3}^H(\cdot,\cdot)$ means Hausdorff distance in $\h^3.$
\end{enumerate}
\end{corollary} 

The proofs of Items (i), (ii) and (iv) follow straightforwardly except for the fact that $X|_{Fr(\Dcal)}:Fr(\Dcal)\to \Mgot$ is injective. However, this can be always guaranteed by trivial refinements of Lemmas \ref{lem:CL} and \ref{lem:ML} and Theorem \ref{th:fun}. Item (ii) generalizes the existence results obtained by the first author in \cite{Al}. 

\end{document}